\documentclass[a4paper]{article}

\usepackage[english]{babel}
\usepackage[utf8]{inputenc}
\usepackage{amsmath}
\usepackage{amssymb}
\usepackage{graphicx}
\usepackage{amsthm}
\usepackage{verbatim}
\usepackage{authblk}

\newtheorem{theorem}{Theorem}
\newtheorem{lemma}{Lemma}
\newtheorem{proposition}{Proposition}
\newtheorem{corollary}{Corollary}

\newtheorem{observation}{Observation}

\newcommand{\hypvert}[1]{\mathcal{V}\mathcal{#1}}
\newcommand{\hypedge}[1]{\mathcal{E}\mathcal{#1}}

\newcommand\ceil[1]{\lceil#1\rceil}
\title{Graph Hausdorff dimension, Kolmogorov complexity and construction of fractal networks}

\author[1]{Leonid Bunimovich}
\author[2,3]{Pavel Skums}
\affil[1]{School of Mathematics, Georgia Institute of Technology,\newline 686 Cherry St NW, Atlanta, GA, USA, 30313}
\affil[2]{Department of Computer Science, Georgia State University,\newline 1 Park Pl NE, Atlanta, GA, USA, 30303}
\affil[3]{Corresponding author. \textit {Email: pskums@gsu.edu}}

\date{\today}

\begin{document}
\maketitle

\begin{abstract}
In this paper we introduce
and study discrete analogues of Lebesgue and Hausdorff dimensions for graphs. It turned out that they are
closely related to well-known graph characteristics such as rank dimension and Prague (or Ne{\v s}et{\v r}il-R{\"o}dl) dimension. It
allows us to formally define fractal graphs and establish fractality of some graph classes. We show, how Hausdorff dimension of graphs is related to their Kolmogorov complexity. We also demonstrate applications of this approach by establishing a novel property of general compact metric spaces using ideas from hypergraphs theory and by proving an estimation for Prague dimension of almost all graphs using methods from algorithmic information theory.

\medskip

\end{abstract}

\section{Introduction}

 Lately there was a growing interest in studying self-similarity and fractal
properties of graphs, which is largely inspired by applications in biology, sociology and chemistry \cite{song2005self,shanker2007defining}. Such studies often employ
statistical physics methods that use ideas from graph theory and general topology, but are not intended to approach the problems under consideration in a rigorous mathematical way. Several studies that translate certain notions of topological dimension theory to graphs using combinatorial methods are also known \cite{smyth2010topological,evako1994dimension}.  However, to the best of our knowledge a rigorous combinatorial theory that defines and studies graph-theoretical analogues of topological fractals still has not been developed. 

In this paper we introduce and study graph analogues of Lebesgue and Hausdorff dimensions of topological spaces from the graph-theoretical point of view. We show that they are closely related to well-known graph characteristics such as rank dimension \cite{berge1984hypergraphs} and Prague (or Ne{\v s}et{\v r}il-R{\"o}dl) dimension  \cite{hell2004graphs}. It
allowed us to define fractal graphs and determine fractality of some graph classes. In particular, it occurred that fractal graphs in some sense could be considered as generalizations of class 2 graphs. We demonstrate that various properties of dimensions of compact topological spaces have graph-theoretical analogues. Moreover, we also show how these relations allow for reverse transfer of combinatorial results to the general topology by proving a new property of general compact metric spaces using machinery from theory of hypergraphs. Finally, we show how Hausdorff (and Prague) dimension of graphs is related to Kolmogorov complexity. This relation allowed us to find a lower bound for Prague dimension for almost all graphs using incompressibility method from the theory of Kolmogorov complexity.

\section{Basic definitions and facts from measure theory, dimension theory and graph theory}
\label{sec:examples}

Let $X$ be a compact metric space. A family $\mathcal{C} = \{C_{\alpha} : \alpha \in A\}$ of open subsets of $X$ is a {\it cover}, if $X = \bigcup_{\alpha \in A} C_{\alpha}$. A cover $\mathcal{C}$ is {\it $k$-cover}, if every $x\in X$ belongs to at most $k$ sets from $\mathcal{C}$; {\it $\epsilon$-cover}, if for every set $C_i\in \mathcal{C}$ its diameter $diam(C_i)$ does not exceed $\epsilon$; $(\epsilon, k)$-cover, if it is both $\epsilon$-cover and $k$-cover. {\it Lebesgue dimension} ({\it cover dimension}) $dim_L(X)$ of the space $X$ is the minimal integer $k$ such that for every $\epsilon > 0$ there exists $(\epsilon, k+1)$-cover of $X$. 

Let $\mathcal{F}$ be a semiring of subsets of a set X. A function $m:\mathcal{F}\rightarrow \mathbb{R}^+_0$ is {\it a measure}, if $m(\emptyset) = 0$ and for any disjoint sets $A,B\in \mathcal{F}$ $m(A\cup B) = m(A) + m(B)$. 

Let now $X$ be a subspace of an Euclidean space $\mathbb{R}^d$. {\it Hyper-rectangle} $R$ is a Cartesian product of semi-open intervals: $R = [a_1,b_1)\times\dots\times [a_d,b_d)$, where $a_i < b_i$, $a_i,b_i \in \mathbb{R}$; the {\it volume} of a the hyper-rectangle $R$ is defined as $vol(R) = \prod_{i=1}^d (b_i - a_i)$. The {\it $d$-dimensional Jordan measure} of the set $X$ is the value
$\mathcal{J}^d(X) = \inf\{\sum_{R\in \mathcal{C}} vol(R)\},$
where infimum is taken over all finite covers $\mathcal{C}$ of $X$ by disjoint hyper-rectangles. The {\it $d$-dimensional Lebesgue measure} of a measurable set $\mathcal{L}^d(X)$ is defined analogously, with the additional condition that infimum  is taken over all countable covers $\mathcal{C}$ of $X$ by (not necessarily disjoint) hyper-rectangles.

Let $s >0$ and $\epsilon > 0$. Consider the parameter
$\mathcal{H}^s_{\epsilon}(X) = \inf\{\sum_{C\in \mathcal{C}}diam(C)^s\},$ where infimum is taken over all $\epsilon$-covers of $X$. The {\it $s$-dimensional Hausdorff measure} of the set $X$ is defined as $\mathcal{H}^s(X) = \lim_{\epsilon \rightarrow 0} \mathcal{H}^s_{\epsilon}(X)$.

The aforementioned measures are related as follows. If Jordan measure of the set $X$ exists, then it is equal to its Lebesgue measure. For Borel sets Lebesgue measure and Hausdorff measure are equivalent in the sense, that for any Borel set $Y$ we have $\mathcal{L}^d(Y) = C_d\mathcal{H}^d(Y)$, where $C_d$ is a constant depending only on $d$.

{\it Hausdorff dimension} $dim_H(X)$ of the set $X$ is the value

\begin{equation}\label{hdimtop}
dim_H(X) = \inf\{s \geq 0 :\mathcal{H}^s(X) < \infty\}.
\end{equation}

Lebesgue and Hausdorff dimension of $X$ are related as follows:

\begin{equation}\label{lebvshaus}
dim_L(X) \leq dim_H(X).
\end{equation}

The set $X$ is {\it a fractal} \cite{falconer2004fractal}, if the inequality (\ref{lebvshaus}) is strict.

Let $G=(V(G),E(G))$ be a simple graph. The family of subgraphs $\mathcal{C} = \{C_1,...,C_m\}$ of $G$ is a {\it cover}, if every edge $uv\in E(G)$ belongs to at least one subgraph from $\mathcal{C}$. A cover $\mathcal{C}$ is {\it $k$-cover}, if every vertex $v\in V(G)$ belongs to at most $k$ subgraphs of $\mathcal{C}$; {\it clique cover}, if all subgraphs $C_i$ are cliques. A set $W\subseteq V(G)$ {\it separates} vertices $u,v\in V(G)$, if $|W\cap \{u,v\}| = 1$.

For a hypergraph $\mathcal{H} = (\mathcal{V}(\mathcal{H}),\mathcal{E}(\mathcal{H}))$, its  {\it rank} $r(\mathcal{H})$ is the maximal size of its edges. A hypergraph $\mathcal{H}$ is {\it strongly $k$-colorable}, if for every vertex a color from the set $\{1,...,k\}$ can be assigned in such a way that vertices of every edge receive different colors.

Intersection graph $L = L(\mathcal{H})$ of a hypergraph $\mathcal{H}$ is a simple graph with a vertex set $V(L) = \{v_E : E\in \mathcal{E}(\mathcal{H})\}$ in a bijective correspondence with the edge set of $\mathcal{H}$ and two distinct vertices $v_E,v_F\in V(L)$ being adjacent, if and only if $E\cap F \ne \emptyset$. The following theorem establishes a connection between intersection graphs and clique $k$-covers:

\begin{theorem}\cite{berge1984hypergraphs}\label{thm:covervsinter}
	A graph $G$ is an intersection graph of a hypergraph of rank $\leq k$ if and only if it has a clique $k$-cover.
\end{theorem}

{\it Rank dimension} \cite{metelsky2003} $dim_R(G)$ of a graph $G$ is the minimal $k$ such that $G$ satisfies conditions of Theorem \ref{thm:covervsinter}.  In particular, graphs with $dim_R(G)=1$ are disjoint unions of cliques (such graphs are called {\it equivalence graphs} \cite{alon1986covering} or {\it $M$-graphs} \cite{tyshkevich1989matr}), and graphs with $dim_R(G)=2$ are line graphs of multigraphs.


{\it Categorical product} of graphs $G_1$ and $G_2$ is the graph $G_1 \times G_2$ with the vertex set $V(G_1\times G_2) = V(G_1)\times V(G_2)$ with two vertices $(u_1,u_2)$ and $(v_1,v_2)$ being adjacent whenever $u_1v_1\in E(G_1)$ and $u_2v_2\in E(G_2)$. {\it Prague dimension} $dim_P(G)$ is the minimal integer $d$ such that $G$ is an induced subgraph of a categorical product of $d$ complete graphs.

{\it Equivalent cover} $\mathcal{M} = \{M_1,...,M_k\}$ of the graph $G$ consists of spanning subgraphs such that each subgraph $M_i$ is an equivalence graph.  Equivalent cover is {\it separating}, if every two distinct vertices of $G$ are separated by one of connected components in some subgraph from $\mathcal{M}$. 

Relations between Prague dimension, clique covers, vertex labeling and intersection graphs are described by the following theorem:

\begin{theorem}\cite{hell2004graphs,babaits1996kmern}\label{thm:pdimcharact}
	The following statements are equivalent:
	
	1) $dim_P(\overline{G}) \leq k$;
	
	2) there exists a separating equivalent cover of $G$;
	
	3) $G$ is an intersection graph of strongly $k$-colorable hypergraph without multiple edges;
    
    4) there exists an injective mapping $\phi: V(G) \rightarrow \mathbb{N}^k$, $v \mapsto (\phi_1(v),\dots,\phi_k(v))$ such that $uv\not\in E(G)$ whenever $\phi_j(u) \ne \phi_j(v)$ for every $j=1,...,k$.
	
\end{theorem}

Numbers of vertices and edges of a graph $G$ are denoted by $n$ and $m$, respectively. A subgraph of $G$ induced by vertex subset $U\subseteq V(G)$ is denoted as $G[U]$. For two graphs $G_1$ and $G_2$ the notation $G_1 \leq G_2$ indicates, that $G_1$ is an induced subgraph of $G_2$. 

\section{Lebesgue dimension of graphs}\label{lebgraph}

Lebesgue dimension of a metric space is defined through $k$-covers by sets of arbitrary small diameter. It is natural to transfer this definition to graphs using graph $k$-covers by subgraphs of smallest possible diameter, i.e. by cliques. Thus by Theorem \ref{thm:covervsinter} we define Lebesgue dimension of a graph through its rank dimension: 

\begin{equation}\label{kdimeqleb}
dim_L(G) = dim_R(G)-1.
\end{equation}

An analogy between Lebesgue and rank dimensions is futher justified by the following Proposition \ref{topspacehyper}, that states that any compact metric spaces of bounded Lebesgue measure could be approximated by intersection graphs of (infinite) hypergraphs of bounded rank. To prove it, we will use the following fact:

\begin{lemma}\label{lebnumber}\cite{edgar2007measure}
Let $X$ be a compact metric space and $\mathcal{U}$ be its open cover. Then there exists $\delta>0$ (called a {\it Lebesgue number} of $\mathcal{U}$) such that for every subset $A\subseteq X$ with $diam(A) < \delta$ there is a set $U\in \mathcal{U}$ such that $A\subseteq U$.
\end{lemma}

\begin{theorem}\label{topspacehyper}
	Let $X$ be a compact metric space with a metric $\rho$. Then $dim_L(X) \leq k-1$ if and only if for any $\epsilon > 0$ there exists a number $0 < \delta < \epsilon$ and a hypergraph $\mathcal{H}(\epsilon)$ on a finite vertex set $V(\mathcal{H}(\epsilon))$ with an edge set $E(\mathcal{H}(\epsilon)) = \{e_x : x\in X\}$ with the following properties: 
	
	1) $rank(\mathcal{H}(\epsilon)) \leq k$;
	
	2) $e_x\cap e_y \ne \emptyset$ for every $x,y\in X$ such that $\rho(x,y) < \delta$;
	
	3) $\rho(x,y) < \epsilon$ for every $x,y\in X$ such that $e_x \cap e_y \ne \emptyset$;
	
	4) for every $v\in V(\mathcal{H}(\epsilon))$ the set $X_v = \{x\in X : v\in e_x \}$ is open.
\end{theorem}

\begin{proof}

The proof borrows some ideas from intersection graphs theory (see \cite{berge1984hypergraphs}). Suppose that $dim_L(X) \leq k$, $\epsilon > 0$ and let $\mathcal{C}$ be the corresponding $(\epsilon,k)$-cover of $X$. Since $X$ is compact, we can assume that $\mathcal{C}$ is finite, i.e. $\mathcal{C} = \{C_1,...,C_m\}$. Let $\delta$ be the Lebesgue number of $\mathcal{C}$.

For a point $x\in X$ let $e_x = \{i\in [m] : x\in C_i\}$. Consider a hypergraph $\mathcal{H}$ with $V(\mathcal{H}) = [m]$ and $E(\mathcal{H}) = \{e_x : x\in X\}$. Then $\mathcal{H}$ satisfies conditions 1)-4). Indeed, $rank(\mathcal{H}) \leq k$, since $\mathcal{C}$ is $k$-cover. If $\rho(x,y) < \delta$, then by Lemma \ref{lebnumber} there is $i\in [m]$ such that $\{x,y\}\in C_i$, i.e. $i\in e_x\cap e_y$. Condition $j\in e_x\cap e_y$ means that $x,y\in C_j$,  and so $\rho(x,y) < \epsilon$, since $diam(C_j) < \epsilon$. Finally, for every $i\in V(\mathcal{H})$ we have $X_v = C_v$, and thus $X_v$ is open.

Conversely, let $\mathcal{H}$ be a hypergraph with $V(\mathcal{H}) = [m]$ satisfying conditions (1)-(4). Then it is straightforward to checked, that $\mathcal{C} = \{X_1,...,X_m\}$ is an open $(\epsilon,k)$-cover of $X$.
\end{proof}

 So, $dim_L(X)\leq k$ whenever for any $\epsilon > 0$ there is a well-defined hypergraph $\mathcal{H}(\epsilon)$ of $rank(\mathcal{H}(\epsilon)) \leq k$ with edges in bijective correspondence with points of $X$ such that two points are close if and only if corresponding edges intersect. 

\section{Graph measure and Hausdorff dimension of graphs}\label{measuregraph}

In order to rigorously define a graph analogue of Hausdorff dimension, we need to define first a corresponding measure. Note that in any meaningful finite graph topology every set is a Borel set. As mentioned above, for measurable Borel sets in $\mathbb{R}^n$ Jordan, Lebesgue and Hausdorff measures are equivalent.  Thus further we will work with a graph analogue of Jordan measure.

It is known, that every graph is isomorphic to an induced subgraph of a categorical product of complete graphs \cite{hell2004graphs}. Consider a graph $G$ embedded into a categorical product of $d$ complete graphs $K^1_{n_1}\times\dots \times K^d_{n_d}$. Without loss of generality we may assume that $n_1 = ... = n_d = n$, i.e.

\begin{equation}\label{graphembed}
G\cong G' \leq S = (K_n)^d,
\end{equation}

Suppose that $V(K^n) = \{1,\dots,n\}$. The graph $S$ will be referred to as {\it a space} of dimension $d$ and $G'$ as an embedding of $G$ into $S$. It is easy to see, that by definition every vertex $v\in S$ is a vector $v = (v_1,...,v_d)$ with $v_i\in [n]$, and two vertices $u$ and $v$ are adjacent in $S$ if and only if $v_r\ne u_r$ for every $r\in [d]$.

{\it Hyper-rectangle} $R = R(J_1,...,J_d)$ is a subgraph of $S$, that is defined as follows: for every $i=1,...,d$ choose a non-empty subset $J_i \subseteq [n]$, then $R = K_n[J_1]\times\dots\times K_n[J_d]$. The {\it volume} of a hyper-rectangle $R$ is the value $vol(R) = |V(R)| = \prod_{i=1}^d |J_i|$.

The family $\mathcal{R} = \{R^1,...,R^m\}$ of hyper-rectangles is a {\it rectangle co-cover} of $G'$, if the subgraphs $R^i$ are pairwise vertex-disjoint, $V(G')\subseteq \bigcup_{i=1}^m V(R^i)$ and $\mathcal{R}$ covers all non-edges of $G'$, i.e. for every $x,y\in V(G')$, $xy\not\in E(G')$ there exists $j\in[m]$ such that $x,y\in V(R^j)$. We define {\it $d$- volume} of a graph $G$ as

\begin{equation}\label{grvol}
vol^d(G) = \min_{G'}\min_{\mathcal{R}} \sum_{R\in \mathcal{R}} vol(R),
\end{equation}

where the first minimum is taken over all embeddings $G'$ of $G$ into $d$-dimensional spaces $S$ and the second minimum - over all rectangle co-covers of $G'$ (see Fig. \ref{fig:embedP4}). 

\begin{figure}[h]
\begin{center}
\includegraphics[width=80mm,height=80mm,keepaspectratio] {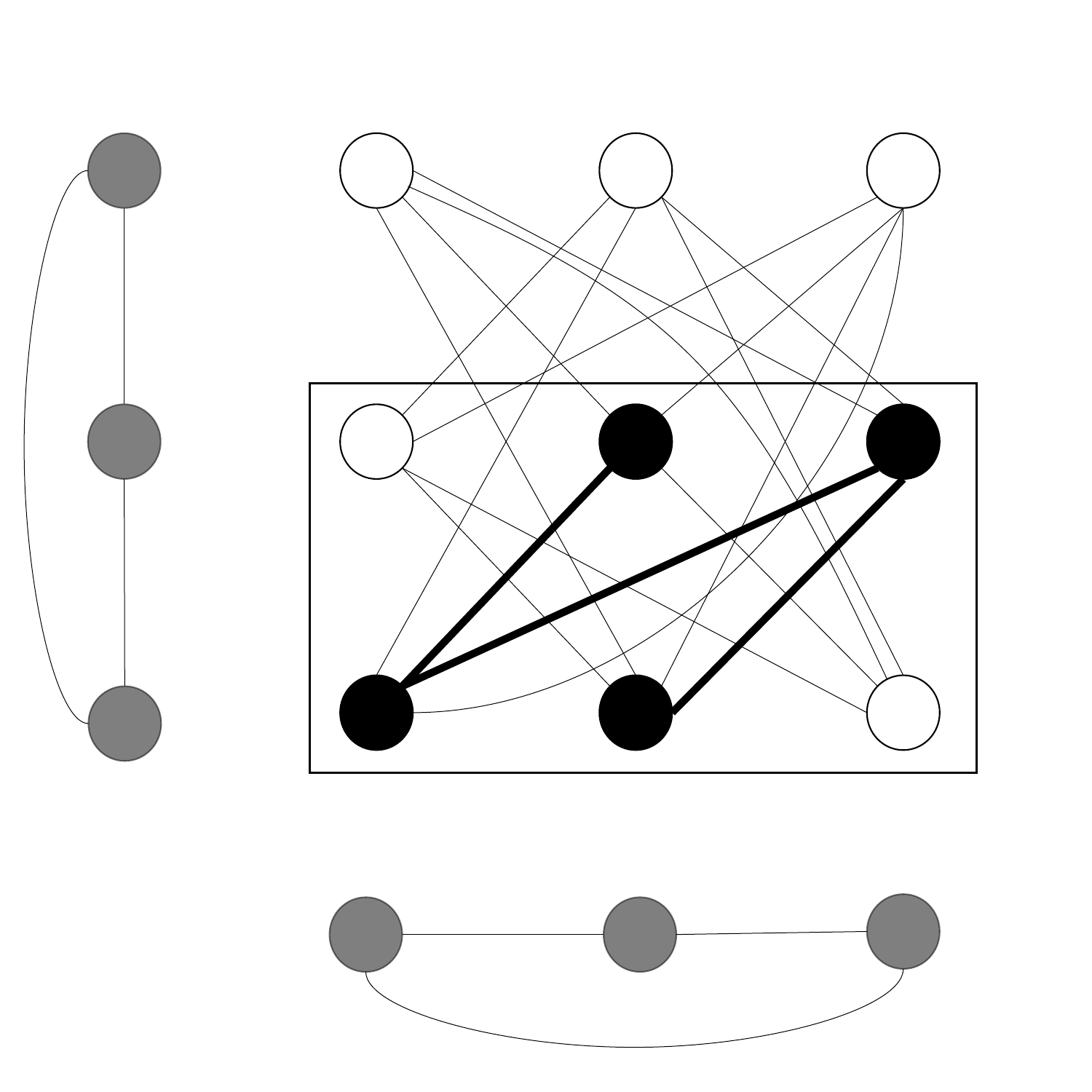}
\caption{\label{fig:embedP4} {\small Embedding of $G=P_4$ into a 2-dimensional space $S=(K_3)^2$ and its rectangle co-cover by a hyper-rectangle of volume $6$}}
\end{center}
\end{figure}

We define a {\it $d$- measure} of a graph $F$ as follows:

\begin{equation}\label{neasuregraph}
\mathcal{H}^d(F) = \left \{
\begin{array}{lll}
vol^d(\overline{F}), \text{ if } \overline{F} \text{ can be represented as (\ref{graphembed})};
\\ +\infty, \text{ otherwise }
\end{array}
\right.
\end{equation}

In the remaining part of this section we will prove, that $\mathcal{H}^d$ indeed satisfy the property of a measure, i.e.
$\mathcal{H}^d(F^1\cup F^2) = \mathcal{H}^d(F^1) + \mathcal{H}^d(F^2)$, where $F^1\cup F^2$ is a disjoint union of graphs $F^1$ and $F^2$.

Let $W^1,W^2\subseteq V(S)$. We write $W_1\sim W_2$, if every vertex from $W_1$ is adjacent to every vertex from $W_2$. Denote by $P_k(W_1)$ {\it the $k$-th projection} of $W_1$, i.e. the set of all $k$-coordinates of vertices of $W_1$: $$J_k(W_1) = \{v_k : v\in W_1\}.$$
In particular, $P_k(R(J_1,...,J_d)) = J_k$.

The following proposition follows directly from the definition of $S$

\begin{proposition}\label{padjcoord}
$W^1 \sim W^2$ if and only if $P_k(W_1)\cap P_k(W_2) = \emptyset$ for every $k\in [d]$.
\end{proposition}

Assume that $\mathcal{R} = \{R^1,...,R^m\}$ is a minimal rectangle co-cover of a minimal embedding $G'$, i.e.  $vol^d(G) = \sum_{R\in \mathcal{R}} vol(R)$.  Further we will demonstrate, that $\mathcal{R}$ has a rather simple structure.

Let $J_k^i = P_k(R^i)$.

\begin{proposition}\label{dvolumeinterindex}
 $J^i_k \cap J^j_k = \emptyset$ for every $i,j\in [m]$, $i\ne j$ and every $k\in [d]$ .
\end{proposition}

\begin{proof}
First note, that for every hyper-rectangle $R^i = R(J^i_1,...,J^i_d)$, every coordinate $k\in [d]$ and every $l\in J^i_k$ there exists $v\in V(G')\cap V(R^i)$ such that $v_k = l$. Indeed, suppose that it does not hold for some $l\in J^i_k$. If $|J_k^i| = 1$, then it means that $V(G')\cap V(R^i) = \emptyset$. Thus, $\mathcal{R}' = \mathcal{R}\setminus \{R^i\}$ is a rectangle co-cover, which contradicts the minimality of $\mathcal{R}$. If $|J_k^i| > 1$,  consider a hyper-rectangle  $(R^i)' = R(J_i^1,...,J_k^i\setminus \{l\},...,J^i_d)$. The set $\mathcal{R}' = \mathcal{R}\setminus \{R^i\}\cup \{(R^i)'\}$  is a rectangle co-cover, and the $d$-volume of $(R^i)'$ is smaller than the $d$-volume of $R^i$. Again it contradicts minimality of $\mathcal{R}$.

Now assume that for some distinct $i,j\in [m]$ and $k\in [d]$ we have $J^i_k \cap J^j_k \supseteq \{l\}$. Then there exist $u\in V(G')\cap V(R^i)$ and $v\in V(G')\cap V(R^j)$ such that $u_k = v_k = l$. So, $uv\not\in E(G')$, and therefore by the definition $uv$ is covered by some $R^h\in \mathcal{R}$. The hyper-rectangle $R^h$ intersects both $R^i$ and $R^j$, which contradicts the definition of a rectangle co-cover.
\end{proof}

\begin{proposition}\label{dvolumecoconcomp}
Let $U^i = V(G')\cap V(R^i)$, $i=1,...,m$. Then the set $U=\{U^1,...,U^m\}$ coincides with the set of co-connected components of $G'$.
\end{proposition}
\begin{proof}
Propositions \ref{padjcoord} and \ref{dvolumeinterindex} imply, that vertices of distinct hyper-rectangles from the co-cover $\mathcal{R}$ are pairwise adjacent. So, $U_i \sim U_j$ for every $i,j\in [m]$, $i\ne j$.

Let $\mathcal{C}=\{C^1,...,C^r\}$ be the set of co-connected components of $G'$ (thus $V(G') = \bigsqcup_{l=1}^r C^l = \bigsqcup_{i=1}^m U^i$). Consider a component $C^l\in \mathcal{C}$ and the sets $C^l_i = C^l \cap U^i$, $i=1,...,m$. We have $C^l = \bigsqcup_{i=1}^m C^l_i$ and $C^l_i \sim C^l_j$ for all $i\ne j$. Therefore, due to co-connectedness of $C^l$, exactly one of the sets $C^l_i$ is non-empty.

So, we have demonstrated, that every co-connected component $C^l$ is contained in some of the sets $U^i$. Now, let some $U_i$ consists of several components, i.e. without loss of generality $U_i = C^1\sqcup\dots\sqcup C^q$, $q\geq 2$. Let $I_k^j = P_k(C^j)$, $j=1,...,q$. By Proposition \ref{padjcoord} we have $I_k^{j_1} \cap I_k^{j_2} = \emptyset$ for all $j_1\ne j_2$, $k=1,...,d$.  Consider hyper-rectangles $R^{i,1} = R(I_1^1,...,I_d^1)$,...,$R^{i,q} = R(I_1^q,...,I_d^q)$. Those hyper-rectangles are pairwise vertex-disjoint, and $C^j \subseteq V(R^{i,j})$ for all $j\in [q]$. Since every pair of non-connected vertices of $G'$ is contained in some of its co-connected components, we arrived to the conclusion, that the set $\mathcal{R}' = \mathcal{R}\setminus \{R^i\} \cup \{R^{i,1},...,R^{i,q}\}$ is a rectangle co-cover. Moreover, $V(R^{i,1})\sqcup\dots\sqcup V(R^{i,q}) \subsetneq V(R^i)$, and therefore $\sum_{j=1}^q vol(R^{i,j}) < vol(R^i)$, which contradicts the minimality of $\mathcal{R}$.
\end{proof}

\begin{corollary}\label{voljoin}
Let $\mathcal{C}=\{C^1,...,C^m\}$ be the set of co-connected components of $G'$. Then $R_i = R(P_1(C^i),...,P_d(C^i))\supseteq C^i$.
\end{corollary}

\begin{proof}
By Proposition \ref{dvolumecoconcomp}, $|\mathcal{R}| = |\mathcal{C}|$, and every component $C^i\in \mathcal{C}$ is contained in a unique hyper-rectangle $R^i\in \mathcal{R}$, $i=1,...,m$. Every pair of non-adjacent vertices of $G'$ belong to some of its co-connected components. This fact, together with the minimality of $\mathcal{R}$, implies that $R^i$ is the minimal hyper-rectangle that contains $C^i$. Thus $R_i = R(P_1(C^i),...,P_d(C^i))$.
\end{proof}

\begin{corollary}\label{cor:minvolspace}
If $\overline{G}$ is connected, than $vol^d(G)$ is the minimal volume of $d$-dimensional space $S$, where $G$ can be embedded. 
\end{corollary}

\begin{corollary}\label{volconcomp}
Let $\mathcal{D}=\{D^1,...,D^m\}$ be the set of co-connected components of $G$. Then $vol^d(G) = \sum_{i=1}^m vol^d(G[D^i])$.
\end{corollary}
\begin{proof}
Suppose that $\{C^1,...,C^m\}$ is the set of co-connected components of $G'$, and $G[D^i]\cong G'[C^i]$. Proposition \ref{dvolumecoconcomp} and Corollary \ref{voljoin} imply that $\{R_i\}$ is a rectangle co-cover of an embedding of $G[D^i]$.  Therefore we have $vol(R^i) \geq vol^d(G[D^i])$ and thus $vol^d(G) = \sum_{i=1}^m vol(R^i) \geq \sum_{i=1}^m vol^d(G[D^i])$.

Now let $G'^i$ be a minimal embedding of $G[D^i]$ into $(K_n)^d$.  By Proposition \ref{dvolumecoconcomp}, every minimal hyper-rectangle co-cover of $G'^i$ consists of a single hyper-rectangle $R^i = R(J^i_1,...,J^i_d)$ or, in other words, $G[D^i]$ is embedded into $R^i$. Now construct an embedding $G'$ of $G$ into $S=(K_{nm})^d$ and its hyper-rectangle co-cover as follows:

let $I^i_k = (i-1)m + J^i_k =  \{(i-1)m + l : l\in J^i_k\}$, $k=1,...,d$. Obviously, $Q^i = K_{mn}(I^i_1)\times\dots\times K_{mn}(I^i_d) \cong R^i$. Now embed $G[D^i]$ into $Q^i$. Let $G'^i$ be those embeddings. By Proposition \ref{padjcoord} $V(G'^i) \sim V(G'^j)$, so $G' = G'^1\cup...\cup G'^m$ is indeed embedding of $G$.

All $Q^i$ are pairwise disjoint. Since every pair of non-adjacent vertices belong to some $G[D^i]$, we have that $\mathcal{Q} = \{Q^1,...,Q^m\}$ is hyper-rectangle co-cover of $G'$. Therefore $vol^d(G) \leq \sum_{i=1}^m vol(Q^i) = \sum_{i=1}^m vol^d(G[D^i])$.
\end{proof}

\begin{theorem}\label{measureadditivity}
Let $F_1$ and $F_2$ be two graphs. Then 

\begin{equation}\label{eq:measaddit}
\mathcal{H}^d(F^1\cup F^2) = \mathcal{H}^d(F^1) + \mathcal{H}^d(F^2)
\end{equation}

\end{theorem}

\begin{proof}
It can be shown \cite{babai1992linear}, that $\overline{F^1\cup F^2}$ can be embedded into a categorical product of $d$ complete graphs if and only if both $\overline{F_1}$ and $\overline{F_2}$ have such embeddings. Therefore the relation (\ref{eq:measaddit}) holds, if one of its members is equal to $+\infty$.

If all $\overline{F_1}$,$\overline{F_2}$, $\overline{F^1\cup F^2}$ can be embedded into a product of $d$ complete graphs, then (\ref{eq:measaddit}) follows from Corollary \ref{volconcomp}.

\end{proof}


Following the analogy with Hausdorff dimension of topological spaces (\ref{hdimtop}), we define a {\it Hausdorff dimension} of a graph $G$ as 

\begin{equation}\label{hdimgraph}
dim_H(G) = \min\{s \geq 0 :\mathcal{H}^s(G) < \infty\}-1.
\end{equation}

Thus, Hausdorff dimension of a graph can be identified with a Prague dimension of its complement minus 1.

\section{Relations with Kolmogorov complexity}

Let $\mathbb{B}^*$ be the set of all finite binary strings and $\Phi:\mathbb{B}^* \rightarrow \mathbb{B}^*$ be a computable function. A {\it Kolmogorov complexity} $K_{\Phi}(s)$ of a binary string $s$ with respect to $\Phi$ is defined as a minimal length of a string $s'$ such as $\Phi(s') = s$. Since Kolmogorov complexities with respect to any two functions differ only by an additive constant \cite{li2009Kolmogorov}, it is usually assumed that some canonical function $\Phi$ is fixed, and Kolmogorov complexity is denoted simply by $K(s)$. Thus, informally $K(s)$ could be described as a length of a shortest encoding of the string $s$, that allows to completely reconstruct it. Analogously, for two strings $s,t\in \mathbb{B}^*$, a {\it conditional Kolmogorov complexity} $K(s|t)$ is a a length of a shortest encoding of $s$, if $t$ is known in advance.  More information on properties of Kolmogorov complexity can be found in \cite{li2009Kolmogorov}.

Every graph $G$ can be naturally encoded using the string representation of an upper triangle of its adjacency matrix. Kolmogorov complexity of a graph could be defined as a Kolmogorov complexity of that string \cite{mowshowitz2012entropy,buhrman1999kolmogorov}.  It gives estimations $K(G) = O(n^2)$, $K(G|n) = O(n^2)$.  Alternatively, $n$-vertex connected labeled graph can be represented as a list of edges with ends of each edge encoded using their binary representations concatenated with a binary representation of $n$. It gives estimations $K(G) \leq 2m\log(n) + \log(n) = O(m\log(n))$, $K(G|n) \leq 2m\log(n) = O(m\log(n))$ \cite{li2009Kolmogorov,mowshowitz2012entropy}.

Further in this section for simplicity we will consider connected graphs (for disconnected graphs all considerations below could be applied to every connected component). Let $dim_H(G) = dim_P(\overline{G}) - 1 = d-1$ and $\mathcal{H}^d(G) = h$. Then by Corollary \ref{cor:minvolspace} $\overline{G}$ is an induced subgraph of a product 
\begin{equation}\label{graphembedmin}
K_{p_1}\times\dots \times K_{p_d},
\end{equation}

where $h = p_1\cdot\dots \cdot p_d$.

So, by Theorem \ref{thm:pdimcharact}, $G$ and $\overline{G}$ could be encoded using a collection of vectors $\phi(v) = (\phi_1(v),\dots,\phi_d(v))$, $v\in V(G)$, $\phi_j(v)\in [p_j]$. Such encoding could be stored as a string containing binary representations of coordinates $\phi_j(v)$ using $\log(p_j)$ bits concatenated with a binary representations of $n$ and $p_j$, $j=1,...,n$. The length of this string is $(n+1)\sum_{j=1}^d \log(p_j) + \log(n)$. Analogously, if $n$ and $p_j$ are given, then the length of encoding is $n\sum_{j=1}^d \log(p_j)$. Thus, the following estimation is true:

\begin{proposition} For any $G$,

\begin{equation}\label{eq:KolG}
K(G)\leq (n+1)\log(\mathcal{H}^d(G)) + \log(n)
\end{equation}

\begin{equation}\label{eq:KolcondG}
K(G|n,p_1,...,p_d)\leq n\log(\mathcal{H}^d(G))
\end{equation}

\end{proposition}

Let $p^* = \max_j p_j$. Then we have $K(G)\leq (n+1)d\log(p^*) + \log(n),$\newline $K(G|n,p_1,...,p_d)\leq nd\log(p^*).$ By minimality of the representation (\ref{graphembedmin}), we have $p^* \leq n$. Thus $K(G) = O(dn\log(n))$, $K(G|n,p_1,...,p_d) = O(dn\log(n))$. So, Hausdorff (and Prague) dimension could be considered as a measure of descriptive complexity of a  graph.  In particular, for graphs with a small Hausdorff dimension (\ref{eq:KolG})-(\ref{eq:KolcondG}) give better estimation of their Kolmogorov complexity than the standard estimations mentioned above.

Relations between Hausdorff (Prague) dimension and Kolmogorov complexity could be used to derive lower bound for Hausdorff (and Prague) dimension in a typical case. More rigorously, let $X$ be a graph property and $\mathcal{P}_n(X)$ be the set of labeled n-vertex graphs having $X$. Property $X$ holds for {\it almost all graphs} \cite{erdHos1977chromatic}, if $|\mathcal{P}_n(X)|/2^{\binom{n}{2}} \rightarrow 1$ as $n \rightarrow \infty$. We will use the following lemma:



\begin{lemma}\label{lem:kolcomplbgr}\cite{buhrman1999kolmogorov}

For every $n > 0$ and $\delta: \mathbb{N}\rightarrow \mathbb{N}$, there are at least $2^{\binom{n}{2}}(1-2^{-\delta(n)})$ $n$-vertex labeled graphs $G$ such that $K(G|n) \geq \frac{n(n-1)}{2} - \delta(n)$.
\end{lemma}

Then the following theorem is true:

\begin{theorem}\label{thm:almostalldim}
For every $\epsilon > 0$, almost all graphs have Prague dimension $d$ such that

\begin{equation}\label{lowbounddim}
d \geq \frac{1}{1+\epsilon}\Big(\frac{n-1}{2\log(n)} - \frac{1}{n}\Big)
\end{equation}
\end{theorem}

\begin{proof}
Let $n_{\epsilon} = \ceil{\frac{2}{\epsilon}}$. Consider a graph $G$ with $n\geq n_{\epsilon}$. From (\ref{eq:KolcondG}) we have $K(G)\leq (n+1)d\log(n) + \log(n)$. Using the fact, that $\frac{1}{n} + \frac{1}{nd}\leq \frac{2}{n}\leq \epsilon$, it is  straightforward to check that $(n+1)d\log(n) + \log(n) \leq (1+\epsilon)nd\log(n)$. Therefore we have

\begin{equation}\label{eq:kolmepsilon}
K(G)\leq (1+\epsilon)nd\log(n)
\end{equation}

Let $X$ be the set of all graphs $G$ such that 

\begin{equation}\label{eq:kolmcondlog}
K(G|n) \geq \frac{n(n-1)}{2} - \log(n)
\end{equation}

Using Lemma \ref{lem:kolcomplbgr} with $\delta(n) = \log(n)$, we conclude that  $|\mathcal{P}_n(X)|/2^{\binom{n}{2}} \geq 1-\frac{1}{n}$, and so almost all graphs have the property $X$.

Now it is easy to see that for graphs with the property $X$ and with $n\geq n_{\epsilon}$ the inequality (\ref{lowbounddim}) holds. It follows by combining inequalities (\ref{eq:kolmepsilon})-(\ref{eq:kolmcondlog}) using the fact that $K(G|n)\leq K(G)$. It concludes the proof.
\end{proof}
Considerations above and proof of Theorem \ref{thm:almostalldim} imply that for every $\epsilon > 0$ and $n$-vertex graph $G$ with sufficiently large $n$, $dim_H(G) \geq  C\frac{K(G)}{n}$, where $C=\frac{1}{(1+\epsilon)\log(p^*)}$. Interestingly, similar relations between Kolmogorov complexity and analogues of Hausdorff dimension hold for other objects. In particular, for Cantor space $\mathcal{C}$ (the space of all infinite 0-1 sequences) it is proved in \cite{mayordomo2002kolmogorov} that Kolmogorov complexity and effective (or constructive) Hausdorff dimension $dim_H(s)$ of each sequence $s$ are related as follows: $dim_H(s)= \liminf\limits_{n \rightarrow \infty} \frac{K(s_n)}{n}$ (here $s_n$ is the prefix of $s$ of length $n$). Similar estimations are known for other variants of Hausdorff dimension \cite{ryabko1994complexity,staiger1993kolmogorov}.

\section{Fractal graphs}\label{fracgraph}

Importantly, the relation (\ref{lebvshaus}) between Lebesgue and Hausdorff dimensions of topological spaces remains true for graphs.

\begin{proposition}\label{lebvshausgraph}
For any graph $G$

\begin{equation}
dim_R(G) - 1 = dim_L(G) \leq dim_H(G) = dim_P(\overline{G})-1.
\end{equation}
\end{proposition}

\begin{proof}
Let Prague dimension of a graph $\overline{G}$ is equal to $k$. Then by Theorem \ref{thm:pdimcharact} $G$ is an intersection graph of strongly $k$-colorable hypergraph. Since rank of every such hypergraph  obviously does not exceed $k$, Theorem \ref{thm:covervsinter} implies, that $dim_R(G)\leq k$.
\end{proof}

Definitions of dimensions immediately imply, that both Lebesgue and Hausdorff dimensions are monotone with respect to induced subgraphs.

 Analogously to the definition of fractals for topological spaces, we say, that a graph $G$ is a {\it fractal}, if $dim_L(G) < dim_H(G)$, i.e. $dim_R(G) < dim_P(\overline{G})$.

The following proposition provides a first non-trivial example of fractal graphs

\begin{proposition}
Triangle-free fractals are exactly triangle-free graphs of class 2
\end{proposition}

\begin{proof}
Note that for triangle-free graphs $dim_R(G) = \Delta(G)$. Moreover, it can be shown (see \cite{hell2004graphs}) that if triangle-free graph $G$ is not a disjoint union of edges $nK_2$, then $dim_P(\overline{G}) = \chi'(G)$, where $\chi'(G)$ is a chromatic index. Therefore triangle-free fractals, that are not disjoint unions of edges, are exactly triangle-free graphs of class 2. On the other hand, all bipartite graphs except $nK_2$ are not fractals, since they all are traingle-free graphs of class 1.
\end{proof}

The connection between fractality and class 2 graphs continue to hold for graphs with maximal degree $\Delta(G) \leq 3$. To formulate the corresponding theorem, we need to introduce the following graph operation: let vertices $a,b,c\in V(G)$ form a triangle. 
Replace the subgraph $G[a,b,c]$ with the graph $H_{12}$ shown on Fig. \ref{fig:H12}, and edges $ua,vb,wc\in E(G)$ (if such edges exist) with edges $ux_a$, $vx_b$, $wx_c$, where $x_a,x_b,x_c$ are vertices of degree 1 of the graph $H_{12}$. Let $\mathcal{T} = \{T_1,...,T_k\}$ be a set of disjoint triangles of the graph $G$. The graph $\widetilde{G}(\mathcal{T})$ is obtained by applying the operation described above to each triangle $T_i$, $i=1,...,k$. 



\begin{figure}[h]
\begin{center}
\includegraphics[width=7cm,height=7cm,keepaspectratio] {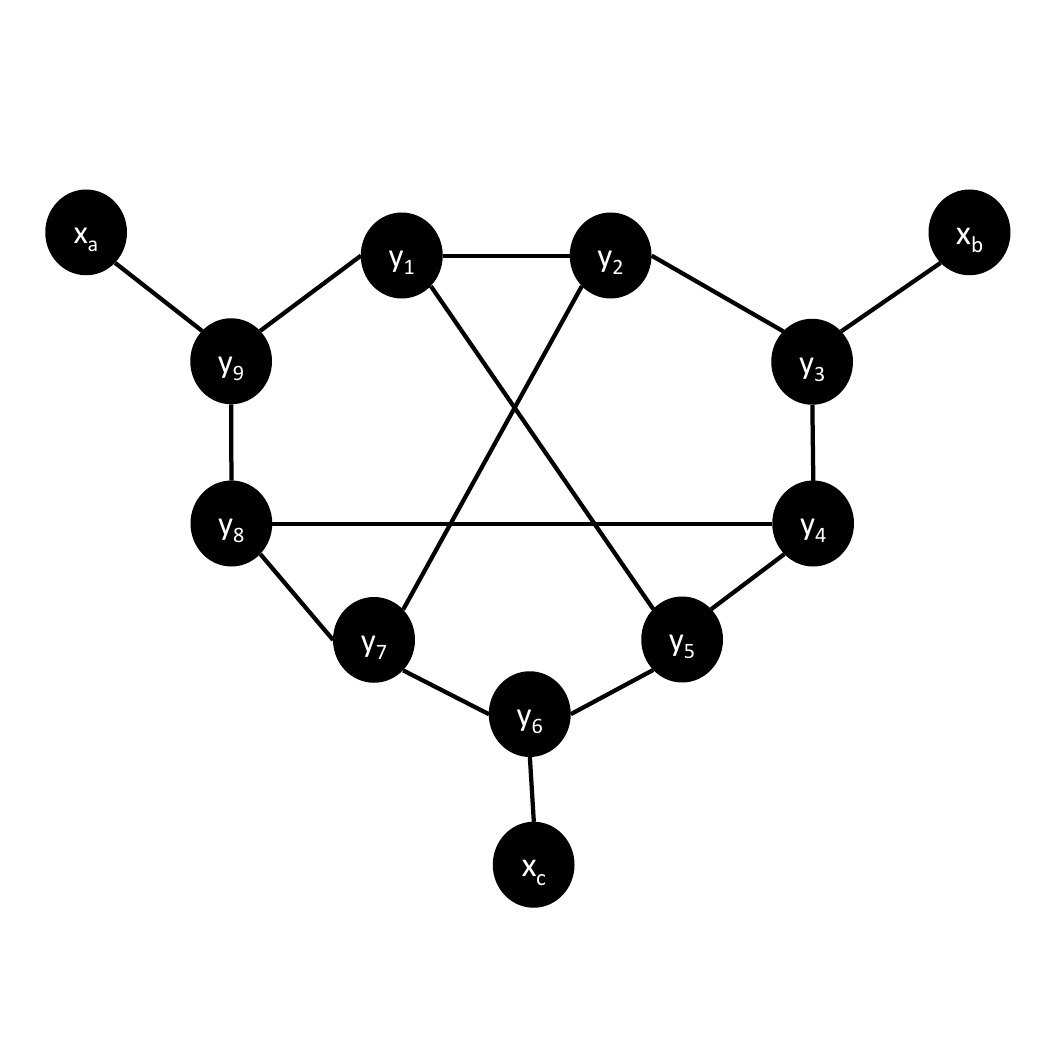}
\caption{\label{fig:H12} {\small Graph $H_{12}$}}
\end{center}
\end{figure}

\begin{theorem}\label{fractaldelta3}
Let $G$ be a connected graph with $\Delta(G) \leq 3$. Then $G$ is a fractal, if and only if one of the following conditions hold:

\begin{itemize}
\item[1)] $G=K_4$
\item[2)] $G$ is claw-free, but contains diamond or odd hole
\item[3)] For any set of disjoint triangles $\mathcal{T}$, the graph $\widetilde{G}(\mathcal{T})$ is of class 2
\end{itemize}
\end{theorem}

\begin{proof}
It is easy to see by definition, that $dim_R(K_4) = 1 < dim_P(O_4) = 2$, which implies that $K_4$ is a fractal. Assume further that $G \ne K_4$. Then obviously, since  $dim_R(G) \leq 3$, $G$ is a fractal if and only if one of the following conditions hold:
\begin{itemize}
\item [(a)] $dim_R(G) = 2$, $dim_P(\overline{G}) = 3$
\item [(b)] $dim_P(\overline{G}) \geq 4$
\end{itemize}

First we will show that the condition (a) is equivalent to the condition 2) from the formulation of the theorem. By Theorem \ref{thm:covervsinter}, $dim_R(G) = 2$ if and only if $G$ is a line graph of a multigraph. Such graphs are characterized by a list of 7 forbidden induced subgraphs \cite{bermond1973representative}, only one of which ($K_{1,3}$) has  the maximal degree, which does not exceed 3 (see e.g. \cite{berge1984hypergraphs}). Therefore $dim_R(G) = 2$ if and only if $G$ is claw-free. 
Furthermore, by Theorem \ref{thm:pdimcharact} $dim_P(\overline{G}) = 2$ if and only if $G$ is a line graph of a bipartite graph. These graphs are exactly (claw,diamond,odd-hole)-free graphs \cite{brandstadt1999graph}. By combining these characterizations together, we get that $dim_R(G) = 2$, $dim_P(\overline{G}) = 3$ if and only if $G$ is claw-free, but contains diamond or odd hole.

Now we will show that the condition (b) is equivalent to the condition 3) from the formulation of the theorem. To do it, we will use the following property of the graph $H_{12}$:

\begin{lemma}\label{h12color}
In every edge 3-coloring of the graph $H_{12}$ edges $x_ay_9,x_by_3,x_cy_6$ have the same color
\end{lemma}
\begin{proof}

Let $c:E(H_{12}) \rightarrow \{1,2,3\}$ be an edge 3-coloring of $H_{12}$. Assume first that $c(y_1y_9) = c(y_5y_6) = 1$, $c(y_1y_2) = c(y_5y_6) = 2$. 

\end{proof}

\end{proof}

As another example of fractal graphs, consider Sierpinski gasket graphs $S_n$ \cite{teguia06serp,klavzar2008coloring,teufl2006spanning}. This class of graphs is associated with the Sierpinski gasket - well-known topological fractal with a Hausdorff dimension $\log(3)/\log(2)\approx 1.585$. Edges of $S_n$ are line segments of the $n$-th approximation of the Sierpinski gasket, and vertices are intersection points of these segments (Fig. \ref{fig:serpGasket}). 

Sierpinski gasket graphs can be defined recursively as follows. Consider tetrads $T_n = (S_n,x_1,x_2,x_3)$, where $x_1,x_2,x_3$ are distinct vertices of $S_n$ called {\it contact vertices}. The first Sierpinski gasket graph $S_1$ is a triangle $K_3$ with vertices $x_1,x_2,x_3$, the first tetrad is defined as $T_1 = (S_1,x_1,x_2,x_3)$. The $(n+1)$-th Sierpinski gasket graph $S_{n+1}$ is constructed from 3 disjoint copies $(S_n,x_1,x_2,x_3)$, $(S'_n,x'_1,x'_2,x'_3)$,  $(S''_n,x''_1,x''_2,x''_3)$ of $n$-th tetrad $T_n$ by gluing together $x_2$ with $x'_1$, $x'_3$ with $x''_2$ and $x_3$ with $x''_1$; the corresponding $(n+1)$-th tetrad is $T_{n+1} = (S_{n+1},x_1,x'_2,x''_3)$.  

\begin{proposition}\label{serpgasket}
For every $n\geq 2$ Sierpinski gasket graph $S_n$ is a fractal with $dim_L(S_n) = 1$ and $dim_H(S_n) = 2$
\end{proposition}

\begin{proof}
First, we will prove that 

\begin{equation}\label{formula:dimssn}
dim_R(S_n)\leq 2, dim_P(\overline{S_n})\leq 3.
\end{equation}

for every $n\geq 2$. We will show it using an induction by $n$. In fact, we will prove slightly stronger fact: for any $n\geq 2$ there exists a clique cover $\mathcal{C} = \{C_1,...,C_m\}$ such that (i) every non-contact vertex is covered by two cliques from $\mathcal{C}$; (ii) every contact vertex is covered by one clique from $\mathcal{C}$; (iii) cliques from $\mathcal{C}$ can be colored using 3 colors in such a way, that intersecting cliques receive different colors and cliques containing different contact vertices also receive different colors; (iv) every two distinct vertices are separated by some clique from $\mathcal{C}$.

For $n=2$ the clique cover $\mathcal{C}$ consisting of 3 cliques, that contain contact vertices, obviously satisfies conditions (i)-(iv) (see Fig. \ref{fig:serpGasket}). Now suppose that $\mathcal{C}$, $\mathcal{C'}$ and $\mathcal{C''}$ are clique covers of $S_n$, $S'_n$ and $S''_n$ with properties (i)-(iv). Assume that $x_i\in C_i$, $x'_i\in C'_i$. $x''_i\in C''_i$ and $C_i,C'_i,C''_i$ have colors $i$, $i=1,...,3$. Then it is straightforward to check, that $\mathcal{C}\cup \mathcal{C'}\cup \mathcal{C''}$ with all cliques keeping their colors is a clique cover of $S_{n+1}$ that satisfies (i)-(iv). So, (\ref{formula:dimssn}) is proved.

Finally, note that for every $n\geq 2$ the graph $S_n$ contains graphs $K_{1,2}$ and $K_4 - e$ as induced subgraphs. Since $dim_R(K_{1,2}) = 2$ and $dim_P(\overline{K_4 - e}) = 3$ (the later is easy to see using clique cover formulation), we have equalities in (\ref{formula:dimssn})
\end{proof}

\begin{figure}[h]
\begin{center}
\includegraphics[width=\textwidth,height=\textheight,keepaspectratio] {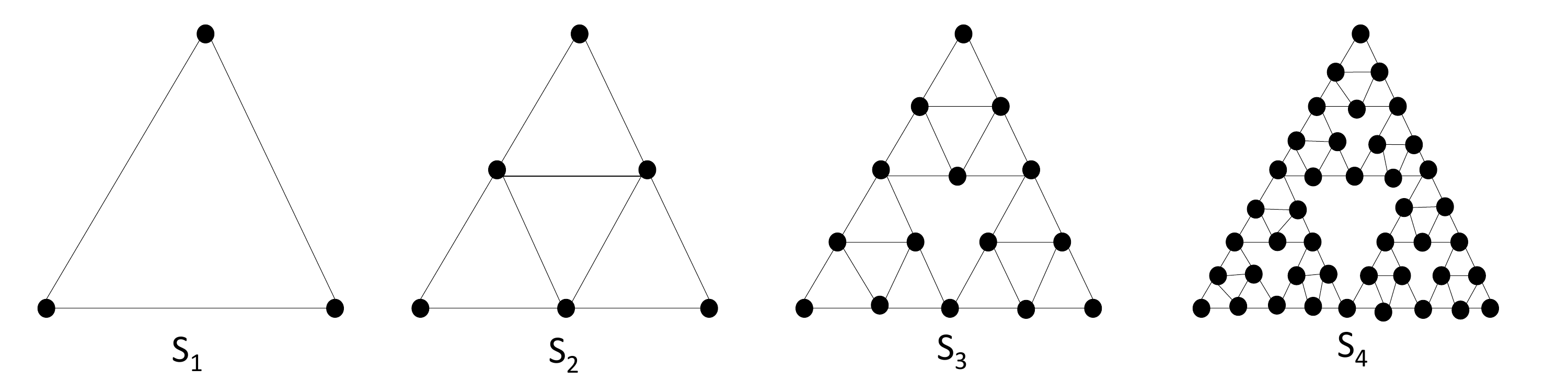}
\caption{\label{fig:serpGasket} {\small Sierpinski gasket graphs $S_1-S_4$}}
\end{center}
\end{figure}

It is important to emphasize that Lebesgue and Hausdorff dimensions of Sierpinski gasket graphs agree with the corresponding dimensions of Sierpinski gasket fractal. As a contrast, note that rectangular grid graphs (cartesian products of 2 paths) are not fractals (by K{\"o}nig theorem since they are bipartite), just like rectangles are not fractals in $\mathbb{R}^2$.

\section{Acknowlegements} The authors were partially supported by the NIH grant 1R01EB025022-01 "Viral Evolution and Spread of Infectious Diseases in Complex Networks: Big Data Analysis and Modeling".

\bibliographystyle{plain}
\bibliography{fracdim_biblio}

\end{document}